\documentclass[11pt,a4paper]{article}

\usepackage{lmodern}
\usepackage{amsmath,amssymb,amsthm,mathtools}
\usepackage{geometry}
\usepackage{xcolor}            
\usepackage[normalem]{ulem}    
\usepackage{hyperref}
\hypersetup{colorlinks=true, linkcolor=blue, citecolor=blue}
\usepackage{enumitem}
\usepackage{booktabs}
\usepackage{authblk}

\mathtoolsset{showonlyrefs}

\newtheorem{theorem}{Theorem}[section]

\newtheorem{lemma}[theorem]{Lemma}
\newtheorem{corollary}[theorem]{Corollary}
\theoremstyle{definition}
\newtheorem{definition}[theorem]{Definition}
\newtheorem{remark}[theorem]{Remark}

\begin{document}

\title{Asymptotic behavior of Eckhoff's method for convergence acceleration
 of Dirac eigenfunction expansions}

\author[1,2]{R.~H.~Barkhudaryan}
\author[1]{G.~G.~Gevorkyan}
\author[1,2,3]{L.~D.~Poghosyan}

\affil[1]{Yerevan State University, Yerevan, Armenia}
\affil[2]{Institute of Mathematics of NAS RA, Yerevan, Armenia}
\affil[3]{DiGREA lab, Center for Scientific Innovation and Education, Yerevan, Armenia}
 % Update placeholder here

\maketitle

\begin{abstract}
    The current paper considers the problem of recovering a vector-function on $[-1,1]$ from a limited number of coefficients of its expansion into a series of eigenfunctions of a one-dimensional Dirac system. The Krylov--Lanczos--Eckhoff--Gottlieb acceleration method is examined in the situation when the boundary values it requires have to be computed from the generalized Fourier coefficients themselves. This leads to a $2q\times 2q$ linear system whose matrix is a block Vandermonde matrix; its determinant and inverse are computed explicitly, and the asymptotic $L_2$-error constant of the method is found, paralleling the classical trigonometric case.
\end{abstract}

\section*{Introduction}

Let $\{\lambda_n\}_{n\in\mathbb{Z}}$, $\{v_n\}_{n\in\mathbb{Z}}$ be the
eigenvalues and $L_2$-orthonormal eigenfunctions of a self-adjoint
boundary value problem on a finite interval, and let $f$ be a
square-integrable function on that interval. It is well known that the
expansion
\[
f(x)=\sum_{n=-\infty}^{\infty}c_n v_n(x),\qquad
c_n=\langle f,v_n\rangle,
\]
converges to $f$ in $L_2$. For practical purposes one approximates
$f$ by the truncated series $S_N(f)=\sum_{|n|\le N}c_n v_n$, and the
involved error is strongly dependent on the smoothness of $f$. If $f$
belongs to the domain of the underlying differential operator, this
approximation converges uniformly. If, however, $f$ or its derivatives
do not satisfy the boundary conditions, the convergence becomes slow,
its rate reflects only the smoothness of $f$, and the Gibbs phenomenon
appears at the endpoints.

For the classical Fourier series ($v_n=e^{i\pi n x}/\sqrt{2}$ on
$[-1,1]$, the eigenfunctions of $i\,d/dx$ with periodic boundary
conditions), increasing the convergence rate by subtracting a
polynomial that carries the discontinuities of the function and of its
first few derivatives was suggested by A.~Krylov as early as 1906
\cite{KR} and later, in 1964, by C.~Lanczos \cite{Lc2,Lc}. More
detailed investigation of this approach, together with practical
algorithms for computing the ``jumps'' of $f$ and its derivatives from
the Fourier coefficients alone, was performed by K.~S.~Eckhoff and
D.~Gottlieb \cite{E1,E2,E3,EW,GG,GSSV}. The method is now usually
called the Krylov--Lanczos--Eckhoff--Gottlieb, or simply Eckhoff,
method. Its asymptotic $L_2$ error constants were computed by A.~Barkhudaryan, R.~Barkhudaryan and A.~Poghosyan in \cite{BBP}. An alternative acceleration technique for nonperiodic
functions, based on quasi-periodic approximations and interpolations,
was developed in \cite{Poghosyan20201,Poghosyan20211,Poghosyan2024}.

For more general self-adjoint boundary value problems, in particular for
Sturm--Liouville and Dirac systems, the analogue of Krylov's correction
polynomial is no longer a Bernoulli polynomial but a combination of
derivatives of the Green's function of the underlying operator. The
basic algebraic identity behind the method, an integration-by-parts
decomposition of $c_n$ into a boundary part $P_n$ and a smoother
interior part $F_n$, was extended by J.~Shaw, L.~Johnson and R.~Riess
\cite{SJR} and, in the form most relevant here, by A.~Nersessian and
A.~Poghosyan \cite{NP1,NP5,NP6}. The Gibbs phenomenon in eigenfunction
expansions of the one-dimensional Dirac system was studied in
\cite{Bgibbs}. For the same system, by R.~Barkhudaryan obtained in
\cite{Bdirac,Baip} a generalization of the Eckhoff method under
the assumption that the boundary values of $f$ and of its successive
Dirac iterates are known \emph{exactly}, and computed the asymptotic
$L_2$ constant of the resulting approximation.

In the current paper we consider the case when these boundary values
are not known in advance and have to be \emph{computed} from the
Fourier coefficients $\{c_n\}_{|n|\le N}$ themselves, exactly as
Eckhoff proposed in the trigonometric case. The boundary conditions on
the eigenfunctions reduce the $4q$ scalar boundary values of $f$ to
$2q$ \emph{reduced boundary values} $g_k^\pm$, and Eckhoff's recipe,
applied at $2q$ selected indices $n_s$, leads to a linear system for
their approximations $\widetilde g_k^\pm$ --- the KEG-D system ---
whose coefficient matrix is a $2q\times 2q$ \emph{block Vandermonde
matrix}, the natural generalization of the $q\times q$ Vandermonde
matrix of the original Eckhoff method. Accordingly, in Section~\ref{sec:dirac} we recall the
Dirac boundary value problem and the exact-jumps form of the
convergence acceleration of \cite{Bdirac}; in Section~\ref{sec:vand}
we prove a closed-form expression for the determinant of the block
Vandermonde matrix and describe its inverse; in Section~\ref{sec:keg-d}
we set up the KEG-D system and estimate the errors
$\widetilde g_k^\pm-g_k^\pm$, showing in particular that strict parity
balance is necessary and sufficient for the leading-order coefficient
matrix to be nonsingular; and in Section~\ref{sec:l2} we prove the
main result of the paper (Theorem~\ref{thm:L2}), an explicit formula
for the limit of
$N^{2q+1}\|\widetilde R_{q,N}(f)\|^2$ in terms of the Dirac jump
constants.

\section{The Dirac system and exact-jumps method}\label{sec:dirac}

We work with the canonical Dirac system on $[-1,1]$
\begin{equation}\label{eq:dirac}
\begin{pmatrix}0 & 1\\-1 & 0\end{pmatrix}\frac{dy}{dx}
-
\begin{pmatrix}p(x) & 0\\ 0 & r(x)\end{pmatrix}y
=
\lambda y,
\end{equation}
with separated self-adjoint boundary conditions
\begin{equation}\label{eq:bc}
y_2(-1)\cos\alpha+y_1(-1)\sin\alpha=0,
\qquad
y_2(1)\cos\beta+y_1(1)\sin\beta=0,
\end{equation}
where $p,r\in C^1[-1,1]$ are real-valued and
$\alpha,\beta\in[0,\pi]$. We assume that $\lambda=0$ is not an
eigenvalue. This entails no loss of generality, since the spectrum is
discrete and a constant spectral shift may always be made so that zero
does not belong to the spectrum.

Let $\{\lambda_n\}_{n\in\mathbb Z}$ be the eigenvalues, indexed in
increasing order, with the origin of the indexing chosen consistently
with the asymptotic formula below, so that
$\lambda_n\to\pm\infty$ as $n\to\pm\infty$. Let $\{v_n\}$ be the
corresponding real orthonormal eigenfunctions in
$L_2([-1,1];\mathbb R^2)$. The asymptotic results of
Levitan--Sargsyan \cite{LS} give
\begin{equation}\label{eq:v-asymp}
\begin{aligned}
v_{n,1}(x) &= K_0\cos(\xi_n(x)-\alpha) + O(1/n),\\
v_{n,2}(x) &= K_0\sin(\xi_n(x)-\alpha) + O(1/n),
\end{aligned}
\end{equation}
as $n\to\pm\infty$, uniformly for $x\in[-1,1]$, where
\begin{equation}\label{eq:xi-def}
\xi_n(x)
:=
\lambda_n(x+1)
+
\frac{1}{2}\int_{-1}^{x}(p(\tau)+r(\tau))\,d\tau.
\end{equation}
Choosing the sign of each eigenfunction so that the leading coefficient
is positive, the $L_2$-normalization determines
\[
K_0=\frac{1}{\sqrt{2}},
\]
since \eqref{eq:v-asymp} yields
\[
\|v_n\|_{L_2}^2
=
\int_{-1}^{1}
\bigl(v_{n,1}^2(x)+v_{n,2}^2(x)\bigr)\,dx
=
2K_0^2+O\!\left(\frac{1}{n}\right)
=
1.
\]

Moreover, the eigenvalues satisfy \cite{LS}
\begin{equation}\label{eq:lam-asymp}
\lambda_n
=
\frac{\pi n}{2}
-
\frac{\widetilde\theta}{2}
+
O\!\left(\frac{1}{n}\right),
\qquad
n\to\pm\infty,
\end{equation}
where
\[
\widetilde\theta
:=
\beta-\alpha
+
\frac{1}{2}\int_{-1}^{1}(p(\tau)+r(\tau))\,d\tau.
\]
%%%%%%%%%%%%%%%%%%%%%%%%%%%%%%%%%%%%%%%%%%%%%%%%%%%%%%%%%%%%%%%%%%%%%%%%%%%%%%%
\subsection{Reduced boundary values}\label{ssec:reduced}

The eigenfunctions satisfy the boundary conditions \eqref{eq:bc}
exactly. Hence, their endpoint values belong to the corresponding
one-dimensional boundary subspaces:
\begin{equation}\label{eq:rank1}
v_n(1)
=
\kappa_n^{(+)}
\binom{\cos\beta}{-\sin\beta},
\qquad
v_n(-1)
=
\kappa_n^{(-)}
\binom{\cos\alpha}{-\sin\alpha},
\end{equation}
for some real scalars $\kappa_n^{(\pm)}$.

We set
\[
\mathcal Af:=Jf'-Qf,
\qquad
J=
\begin{pmatrix}
0&1\\
-1&0
\end{pmatrix},
\qquad
Q(x)=\operatorname{diag}(p(x),r(x)),
\]
and define
\begin{equation}\label{eq:B-fk-def}
\mathcal Bf:=-Jf=\binom{-f_2}{f_1},
\qquad
f_k:=\mathcal B\mathcal A^k f=-J\mathcal A^k f,
\qquad
k\geq0.
\end{equation}
Thus,
\begin{equation}\label{eq:fk-components}
\mathcal A^k f
=
\binom{(\mathcal A^k f)_1}{(\mathcal A^k f)_2}, \quad f_k
=
\binom{-(\mathcal A^k f)_2}{(\mathcal A^k f)_1}.
\end{equation}

For the remainder of this subsection, let $q\geq1$ and assume in
addition that
\[
f\in C^q([-1,1];\mathbb R^2),
\qquad
p,r\in C^q[-1,1],
\]
so that $\mathcal A^k f$ is continuous for $k=0,\dots,q$. For
$k=0,\dots,q$, we define the \emph{reduced boundary values}
\begin{equation}\label{eq:gk}
g_k^\pm
:=
f_{k,1}(\pm1)\cos\gamma_\pm
-
f_{k,2}(\pm1)\sin\gamma_\pm,
\qquad
\gamma_+=\beta,
\qquad
\gamma_-=\alpha.
\end{equation}
Equivalently,
\begin{equation}\label{eq:gk-components}
g_k^\pm
=
-(\mathcal A^k f)_2(\pm1)\cos\gamma_\pm
-
(\mathcal A^k f)_1(\pm1)\sin\gamma_\pm.
\end{equation}

By \eqref{eq:rank1}, the boundary terms involving the eigenfunctions
reduce to
\[
\begin{aligned}
v_n^T(\pm1)f_k(\pm1)
&=
\kappa_n^{(\pm)}
\begin{pmatrix}
\cos\gamma_\pm & -\sin\gamma_\pm
\end{pmatrix}
\binom{f_{k,1}(\pm1)}{f_{k,2}(\pm1)}
\\
&=
\kappa_n^{(\pm)}g_k^\pm.
\end{aligned}
\]

It remains to determine the asymptotic behavior of
$\kappa_n^{(\pm)}$. By \eqref{eq:xi-def} and \eqref{eq:lam-asymp},
\[
\xi_n(1)-\alpha
=
n\pi-\beta
+
O\!\left(\frac{1}{n}\right),
\qquad
\xi_n(-1)-\alpha
=
-\alpha.
\]
Substituting these relations into \eqref{eq:v-asymp} gives
\[
v_n(1)
=
K_0(-1)^n
\binom{\cos\beta}{-\sin\beta}
+
O\!\left(\frac{1}{n}\right),
\]
and
\[
v_n(-1)
=
K_0
\binom{\cos\alpha}{-\sin\alpha}
+
O\!\left(\frac{1}{n}\right).
\]
Since the boundary-direction vectors in \eqref{eq:rank1} have unit
norm, taking the scalar product with the corresponding direction in
the above asymptotic formulas and using \eqref{eq:rank1} yields
\begin{equation}\label{eq:kappapm-asymp}
\kappa_n^{(+)}
=
K_0(-1)^n
+
O\!\left(\frac{1}{n}\right),
\qquad
\kappa_n^{(-)}
=
K_0
+
O\!\left(\frac{1}{n}\right).
\end{equation}

%%%%%%%%%%%%%%%%%%%%%%%%%%%%%%%%%%%%%%%%%%%%%%%%%%%%%%%%%%%%%%%%%%%%%%%%%%%
\subsection{Approximation with exact boundary corrections}
\label{ssec:exact-boundary-corrections}

For $q\geq1$, let
\[
f\in C^q([-1,1];\mathbb R^2),
\qquad
p,r\in C^q[-1,1].
\]
Define the generalized Fourier coefficients of $f$ with respect to the
orthonormal Dirac eigenfunctions $\{v_n\}_{n\in\mathbb Z}$ by
\[
c_n
:=
\langle f,v_n\rangle_{L_2}
=
\int_{-1}^{1}v_n^T(x)f(x)\,dx.
\]

We first derive a one-step integration-by-parts identity. Since
$\mathcal Av_n=\lambda_n v_n$, for any sufficiently regular
vector-valued function $h$,
\[
\lambda_n\int_{-1}^{1}v_n^T h\,dx
=
\int_{-1}^{1}(\mathcal Av_n)^T h\,dx.
\]
Using \(\mathcal Av_n=Jv_n'-Qv_n\), \(J^T=-J\), \(Q^T=Q\),
we obtain
\[
\begin{aligned}
\int_{-1}^{1}(\mathcal Av_n)^T h\,dx
&=
\int_{-1}^{1}
\left(-v_n'^T Jh-v_n^TQh\right)\,dx
\\
&=
-\bigl[v_n^T Jh\bigr]_{-1}^{1}
+
\int_{-1}^{1}v_n^T(Jh'-Qh)\,dx
\\
&=
v_n^T(-1)Jh(-1)
-
v_n^T(1)Jh(1)
+
\int_{-1}^{1}v_n^T\mathcal Ah\,dx.
\end{aligned}
\]
Therefore,
\begin{equation}\label{eq:ibp-one-step}
\int_{-1}^{1}v_n^T h\,dx
=
\lambda_n^{-1}
\left[
v_n^T(-1)Jh(-1)
-
v_n^T(1)Jh(1)
\right]
+
\lambda_n^{-1}
\int_{-1}^{1}v_n^T\mathcal Ah\,dx.
\end{equation}
We now apply \eqref{eq:ibp-one-step} with $h=\mathcal A^k f$.
Using \eqref{eq:B-fk-def}, \eqref{eq:rank1} and the definition \eqref{eq:gk},
\[
\begin{aligned}
v_n^T(-1)J\mathcal A^k f(-1)
&-
v_n^T(1)J\mathcal A^k f(1)
\\
&
=
-v_n^T(-1)f_k(-1)
+
v_n^T(1)f_k(1)
\\
&
=
-\kappa_n^{(-)}g_k^-
+
\kappa_n^{(+)}g_k^+.
\end{aligned}
\]
Define
\[
x_n:=\frac{1}{\lambda_n}.
\]
Then
\begin{equation}\label{eq:ibp-recursion}
\int_{-1}^{1}v_n^T\mathcal A^k f\,dx
=
x_n\bigl(\kappa_n^{(+)}g_k^+-\kappa_n^{(-)}g_k^-\bigr)
+
x_n\int_{-1}^{1}v_n^T\mathcal A^{k+1}f\,dx.
\end{equation}

Starting with $k=0$ and iterating
\eqref{eq:ibp-recursion} $q$ times yields
\begin{equation}\label{eq:cn-decomp}
c_n=P_n+F_n,
\end{equation}
where
\begin{equation}\label{eq:PnFn}
P_n
=
x_n\sum_{k=0}^{q-1}
x_n^k\bigl(\kappa_n^{(+)}g_k^+-\kappa_n^{(-)}g_k^-\bigr),
\qquad
F_n
=
\lambda_n^{-q}
\int_{-1}^{1}
v_n^T(x)\mathcal A^q f(x)\,dx.
\end{equation}

If, in addition, $f^{(q)}$ is absolutely continuous, one
more integration by parts, combined with the asymptotics
\eqref{eq:v-asymp} and the Riemann--Lebesgue theorem (cf.\
\cite{Bdirac}), gives the sharper expansion
\begin{equation}\label{eq:Fn-asymp}
F_n=x_n^{q+1}\bigl[\kappa_n^{(+)}g_q^+-\kappa_n^{(-)}g_q^-\bigr]+o(|n|^{-q-1}),
\qquad |n|\to\infty.
\end{equation}

Following \cite{Bdirac} we define the exact-jumps approximation
\begin{equation}\label{eq:Sqn-exact}
S_{q,N}(f) := P(x) + \sum_{|n|\le N} F_n v_n(x),
\end{equation}
\begin{equation*}
P(x):=\sum_{k=0}^{q-1}\bigl[G_k(x,1,0)f_k(1)-G_k(x,-1,0)f_k(-1)\bigr],
\end{equation*}
where
\[
G_k(x,\xi,0)
=
\frac{1}{k!}\partial_\lambda^k G(x,\xi,\lambda)
\big|_{\lambda=0},
\]
with $G$ denoting the Green's matrix of
\eqref{eq:dirac}--\eqref{eq:bc}. By the spectral representation of the Green's matrix and its
derivatives, the generalized Fourier coefficient of $P$ corresponding
to $v_n$ is precisely $P_n$, that is,
\[
\int_{-1}^{1}v_n^T(x)P(x)\,dx=P_n.
\]
Substituting \eqref{eq:kappapm-asymp} into \eqref{eq:Fn-asymp} and using
the reduced boundary values $g_q^\pm$ defined in \eqref{eq:gk},
\begin{equation}\label{eq:Fn-AB}
F_n=K_0\,x_n^{q+1}\bigl[(-1)^n A+B\bigr]+o(|n|^{-q-1}),
\end{equation}
with the \emph{Dirac jump constants}
\begin{equation}\label{eq:AB}
\begin{aligned}
A&:=g_q^+ = f_{q,1}(1)\cos\beta - f_{q,2}(1)\sin\beta,\\
B&:=-g_q^- = f_{q,2}(-1)\sin\alpha - f_{q,1}(-1)\cos\alpha.
\end{aligned}
\end{equation}
Hence, by \eqref{eq:cn-decomp}, the generalized Fourier coefficients of
$f-P$ are precisely $F_n$. Consequently,
\[
R_{q,N}(f)
=
f-S_{q,N}(f)
=
\sum_{|n|>N}F_n v_n.
\]
By Parseval and \eqref{eq:Fn-AB},
\begin{equation}\label{eq:exact-L2}
\lim_{N\to\infty}N^{2q+1}\|R_{q,N}(f)\|^2
=\frac{2^{2q+2}}{\pi^{2q+2}(2q+1)}(A^2+B^2),
\end{equation}
where the constant is obtained by a Riemann-sum argument applied to
$$\sum_{|n|>N}K_0^2 x_n^{2q+2}[(-1)^nA+B]^2,$$ split by parity (see
the proof of Theorem~\ref{thm:L2}), using $K_0^2=1/2$. The formula
\eqref{eq:exact-L2} replaces Theorem~2 of \cite{Bdirac}: the constant
$2^{2q+3}/\pi^{2q+3}$ stated there must be replaced by
$2^{2q+2}/\pi^{2q+2}$, in accordance with $K_0^2=1/2$ rather than
$1/\pi$, and the boundary directions entering the constant $A$ there
involve an extraneous angle $2\theta-\beta$; the corrected
constants are \eqref{eq:AB}.

The approximation \eqref{eq:Sqn-exact} has a serious drawback: it assumes the boundary values $f_k(\pm1)$ to be known in advance, while in many applications they are not. In the next two sections we develop the tools needed to compute them from the generalized Fourier coefficients themselves.

%%%%%%%%%%%%%%%%%%%%%%%%%%%%%%%%%%%%%%%%%%%%%%%%%%%%%%%%%%%%%%%%%%%%%%%%%%%%%
\section{Block Vandermonde matrices}\label{sec:vand}

In this section, we study the algebraic structure of the coefficient
matrix arising in the linear system used to approximate the reduced
boundary values.

\begin{definition}
Let $q\in\mathbb N$, and let
$\alpha_1,\dots,\alpha_{2q}$, $\beta_1,\dots,\beta_{2q}$, and
$x_1,\dots,x_{2q}$ be complex numbers. We define the
\emph{block Vandermonde matrix of order $2q$} to be the following
$2q\times 2q$ matrix:
\begin{equation}\label{eq:bv}
\Lambda
=
\begin{pmatrix}
\alpha_1 & \alpha_1x_1 & \cdots & \alpha_1x_1^{q-1}
&
\beta_1 & \beta_1x_1 & \cdots & \beta_1x_1^{q-1}
\\
\alpha_2 & \alpha_2x_2 & \cdots & \alpha_2x_2^{q-1}
&
\beta_2 & \beta_2x_2 & \cdots & \beta_2x_2^{q-1}
\\
\vdots & \vdots & \ddots & \vdots
&
\vdots & \vdots & \ddots & \vdots
\\
\alpha_{2q} & \alpha_{2q}x_{2q} & \cdots
& \alpha_{2q}x_{2q}^{q-1}
&
\beta_{2q} & \beta_{2q}x_{2q} & \cdots
& \beta_{2q}x_{2q}^{q-1}
\end{pmatrix}.
\end{equation}
\end{definition}

Let
\[
V_q
=
\bigl(x_i^{j-1}\bigr)_
{1\le i\le 2q,\;1\le j\le q}
\]
denote the rectangular $2q\times q$ Vandermonde matrix, and let
\[
D_\alpha
=
\operatorname{diag}(\alpha_1,\dots,\alpha_{2q}),
\qquad
D_\beta
=
\operatorname{diag}(\beta_1,\dots,\beta_{2q}).
\]
Then
\begin{equation}\label{eq:bv-factor}
\Lambda
=
\bigl(D_\alpha V_q\mid D_\beta V_q\bigr)
=
BC,
\qquad
B
=
\bigl(D_\alpha\mid D_\beta\bigr),
\qquad
C
=
\begin{pmatrix}
V_q & 0\\
0 & V_q
\end{pmatrix}.
\end{equation}
Here,
\[
B\in\mathbb C^{2q\times 4q},
\qquad
C\in\mathbb C^{4q\times 2q}.
\]

For a subset $T\subseteq[2q]:=\{1,\dots,2q\}$, write
\[
T=\{t_1<\cdots<t_m\},
\qquad
m=|T|,
\]
and define the square Vandermonde matrix associated with $T$ by
\[
V_T
=
\bigl(x_{t_i}^{j-1}\bigr)_{1\le i,j\le m}.
\]
Its determinant is
\[
\det(V_T)
=
\prod_{1\le i<j\le m}
\bigl(x_{t_j}-x_{t_i}\bigr).
\]
The factorization \eqref{eq:bv-factor} allows us to compute the
determinant of $\Lambda$ systematically by means of the
Cauchy--Binet formula.

\subsection{Determinant}

\begin{theorem}\label{thm:det}
The determinant of the block Vandermonde matrix \eqref{eq:bv} is
\begin{equation}\label{eq:block-vandermonde-det}
\det\Lambda
=
\sum_{\substack{T\subseteq[2q]\\ |T|=q}}
\varepsilon(T)
\left(\prod_{i\in T}\alpha_i\right)
\left(\prod_{j\in T^c}\beta_j\right)
\det(V_T)\det(V_{T^c}),
\end{equation}
where, for each
\begin{equation}\label{TT^c}
T=\{t_1<\cdots<t_q\},
\qquad
T^c=[2q]\setminus T=\{u_1<\cdots<u_q\},
\end{equation}
we set
\[
\varepsilon(T)
:=
(-1)^{\sum_{k=1}^q(t_k-k)}.
\]
Equivalently,
\[
\det\Lambda
=
\sum_{\substack{T\subseteq[2q]\\ |T|=q}}
\varepsilon(T)
\left(\prod_{i\in T}\alpha_i\right)
\left(\prod_{j\in T^c}\beta_j\right)
\prod_{1\le i<j\le q}
\bigl(x_{t_j}-x_{t_i}\bigr)
\prod_{1\le i<j\le q}
\bigl(x_{u_j}-x_{u_i}\bigr).
\]
\end{theorem}

\begin{proof}
Applying the Cauchy--Binet formula (see, for example, \cite{Gantmacher1959}) to the factorization
\(\Lambda=BC\) in \eqref{eq:bv-factor}, we have
\[
\det\Lambda
=
\sum_{\substack{S\subseteq[4q]\\ |S|=2q}}
\det(B_{[:,S]})\det(C_{[S,:]}),
\]
where \(B_{[:,S]}\) denotes the submatrix of \(B\) formed by all rows
and the columns indexed by \(S\), while \(C_{[S,:]}\) denotes the
submatrix of \(C\) formed by the rows indexed by \(S\) and all columns.

A subset $S\subseteq[4q]$ with $|S|=2q$ can contribute only if
$B_{[:,S]}$ is nonsingular. Each column of \(B\) is a scalar multiple
of a standard basis vector \(\mathbf e_i\): column \(i\) is
\(\alpha_i\mathbf e_i\), while column \(2q+i\) is
\(\beta_i\mathbf e_i\). Hence \(B_{[:,S]}\) is nonsingular only if
exactly one column is selected from each pair
\(\{i,2q+i\}\), \(i=1,\dots,2q\).

Define
\[
T:=S\cap[2q].
\]
Since exactly one element is selected from each pair
\(\{i,2q+i\}\), \(i\in[2q]\), we have
\[
T^c:=[2q]\setminus T
=
\{i\in[2q]:2q+i\in S\}.
\]
Since the indices in the first block of \(C\) precede those in the
second block, the submatrix \(C_{[S,:]}\) has the block-diagonal form
\[
C_{[S,:]}
=
\begin{pmatrix}
\bigl(x_i^{j-1}\bigr)_{\substack{i\in T\\1\leq j\leq q}} & 0\\
0 &
\bigl(x_i^{j-1}\bigr)_{\substack{i\in T^c\\1\leq j\leq q}}
\end{pmatrix}.
\]
If $C_{[S,:]}$ is nonsingular, then its first $q$ columns must be
linearly independent. Since these columns are supported only in the
first $|T|$ rows, this requires $|T|\ge q$. Likewise, the last $q$
columns are supported only in the remaining $2q-|T|$ rows, so their
linear independence requires
\[
2q-|T|\ge q.
\]
Hence $|T|\le q$, and therefore $|T|=q$. Thus both diagonal blocks are the square Vandermonde matrices
\(V_T\) and \(V_{T^c}\).

For each \(q\)-element subset \(T\subseteq[2q]\), the corresponding
set \(S\) is uniquely determined. If
\[
T=\{t_1<\cdots<t_q\},
\qquad
T^c=\{u_1<\cdots<u_q\},
\]
then the selected columns of \(B\) occur in the order
\[
\alpha_{t_1}\mathbf e_{t_1},\dots,
\alpha_{t_q}\mathbf e_{t_q},
\beta_{u_1}\mathbf e_{u_1},\dots,
\beta_{u_q}\mathbf e_{u_q}.
\]
The sign of the corresponding permutation is
\[
\varepsilon(T)
=
(-1)^{\sum_{k=1}^q(t_k-k)}.
\]
Hence
\[
\det(B_{[:,S]})
=
\varepsilon(T)
\left(\prod_{i\in T}\alpha_i\right)
\left(\prod_{j\in T^c}\beta_j\right),
\]
while
\[
\det(C_{[S,:]})
=
\det(V_T)\det(V_{T^c}).
\]

Summing over all subsets \(T\subseteq[2q]\) with \(|T|=q\) gives
\eqref{eq:block-vandermonde-det}.
\end{proof}

\begin{corollary}\label{cor:invert}
The matrix $\Lambda$ is
invertible if and only if the right-hand side of \eqref{eq:block-vandermonde-det} is
non-zero. In particular, $\Lambda$ is singular whenever $\alpha_i=\beta_i$
for all~$i$.
\end{corollary}

\begin{proof}
The invertibility criterion follows immediately from
Theorem~\ref{thm:det}. If $\alpha_i=\beta_i$ for
all~$i$, then the $j$-th columns of the two blocks in \eqref{eq:bv}
coincide, so $\Lambda$ is singular. (Note that this is not visible
term by term in \eqref{eq:block-vandermonde-det}: the terms of $T$ and $T^c$ cancel in
pairs only for odd $q$, since
$\varepsilon(T^c)=(-1)^{q}\varepsilon(T)$.)
\end{proof}

\subsection{Inverse via polynomial interpolation}

\begin{theorem}\label{thm:inv-interp}
Suppose that $\det(\Lambda)\ne0$, and write
\[
\Lambda^{-1}
=
\begin{pmatrix}
Y\\
Z
\end{pmatrix},
\qquad
Y,Z\in\mathbb C^{q\times 2q}.
\]
For each $j=1,\dots,2q$, define the polynomials
\[
Y_j(x)
:=
\sum_{k=0}^{q-1}Y_{k+1,j}x^k,
\qquad
Z_j(x)
:=
\sum_{k=0}^{q-1}Z_{k+1,j}x^k.
\]
Then $(Y_j,Z_j)$ is the unique pair of polynomials of degree at most
$q-1$ satisfying
\begin{equation}\label{eq:interp}
\alpha_iY_j(x_i)+\beta_iZ_j(x_i)
=
\delta_{ij},
\qquad
i=1,\dots,2q.
\end{equation}
\end{theorem}

\begin{proof}
Let $\mathbf y_j$ and $\mathbf z_j$ denote the $j$-th columns of
$Y$ and $Z$, respectively. The $j$-th column of the identity
$\Lambda\Lambda^{-1}=I_{2q}$ gives
\[
D_\alpha V_q\mathbf y_j
+
D_\beta V_q\mathbf z_j
=
\mathbf e_j.
\]
By the definitions of $Y_j$ and $Z_j$, the $i$-th component of this
identity is
\[
\alpha_iY_j(x_i)+\beta_iZ_j(x_i)
=
\delta_{ij},
\]
which proves \eqref{eq:interp}.

Conversely, any pair of polynomials of degree at most $q-1$
satisfying \eqref{eq:interp} determines coefficient vectors
$\mathbf y,\mathbf z\in\mathbb C^q$ satisfying
\[
\Lambda
\binom{\mathbf y}{\mathbf z}
=
\mathbf e_j.
\]
Since $\Lambda$ is invertible, this system has a unique solution.
Hence the pair $(Y_j,Z_j)$ is unique.
\end{proof}

%%%%%%%%%%%%%%%%%%%%%%%%%%%%%%%%%%%%%%%%%%%%%%%%%%%%%%%%%%%%%%%%%%%%%%%%
\subsection{An explicit cofactor formula}

For \(j\in[2q]\), let
\[
R_j:=[2q]\setminus\{j\},
\]
equipped with the natural order inherited from \([2q]\). Define the order
isomorphism
\[
\rho_j:R_j\to[2q-1]
\]
by
\[
\rho_j(i):=
\begin{cases}
i, & i<j,\\
i-1, & i>j.
\end{cases}
\]
For a subset
\[
T=\{t_1<\cdots<t_s\}\subseteq R_j,
\qquad
s=|T|,
\]
let
\[
T':=R_j\setminus T.
\]
We define the shuffle sign \(\varepsilon_j(T)\) by
\begin{equation}\label{eq:epsilon-j}
\varepsilon_j(T)
:=
(-1)^{\sum_{k=1}^s(\rho_j(t_k)-k)}
=
(-1)^{\#\{(u,t)\in T'\times T:\ u<t\}}.
\end{equation}
Equivalently, $\varepsilon_j(T)$ is the sign of the shuffle permutation
of the ordered set $R_j$ that places the elements of $T$ before the
elements of $T'$, while preserving the natural order within each set.

For any index set
\[
S=\{s_1<\cdots<s_d\}\subseteq[2q]
\]
and exponent set
\[
E=\{e_1<\cdots<e_d\}\subset\mathbb N_0,
\]
we denote the generalized Vandermonde matrix by
\[
V_S^E
:=
\bigl(x_{s_i}^{e_k}\bigr)_{1\leq i,k\leq d}.
\]

For variables $y_1,\dots,y_d$, let
\[
e_r(y_1,\dots,y_d)
=
\sum_{1\le i_1<\cdots<i_r\le d}
y_{i_1}\cdots y_{i_r}
\]
denote the elementary symmetric polynomial of degree $r$, with the
conventions
\[
e_0=1,
\qquad
e_r=0 \quad \text{for } r<0 \text{ or } r>d.
\]
For $T\subseteq[2q]$, we write
\[
e_r\!\left((x_t)_{t\in T}\right)
\]
for the elementary symmetric polynomial of degree $r$ in the variables
$x_t$, $t\in T$.

We also adopt the standard convention that empty products and
determinants of \(0\times0\) matrices equal \(1\).

For \(r=1,\dots,q\), set
\[
E_r:=\{0,1,\dots,q-1\}\setminus\{r-1\}.
\]

For \(1\leq j,m\leq2q\), let \(M_{j,m}\) denote the minor obtained
from \(\Lambda\) by deleting row \(j\) and column \(m\). If
\(\det(\Lambda)\neq0\), then the standard cofactor formula gives
\begin{equation}\label{eq:inv-entry}
(\Lambda^{-1})_{m,j}
=
\frac{(-1)^{j+m}}{\det(\Lambda)}\,M_{j,m}.
\end{equation}

\begin{theorem}\label{thm:inv-explicit}
Fix \(j\in[2q]\).

If \(1\leq m\leq q\), then
\begin{equation}\label{eq:Mjm-alpha}
M_{j,m}
=
\sum_{\substack{T\subseteq R_j\\ |T|=q-1}}
\varepsilon_j(T)
\left(\prod_{i\in T}\alpha_i\right)
\left(\prod_{i\in T'}\beta_i\right)
\det\!\left(V_T^{E_m}\right)
\det(V_{T'}),
\end{equation}
where
\begin{equation}\label{eq:generalized-vandermonde-alpha}
\det\!\left(V_T^{E_m}\right)
=
\prod_{\substack{r<s\\ r,s\in T}}(x_s-x_r)\,
e_{q-m}\!\left((x_t)_{t\in T}\right).
\end{equation}

If \(m=q+\ell\), \(1\leq\ell\leq q\), then
\begin{equation}\label{eq:Mjm-beta}
M_{j,q+\ell}
=
\sum_{\substack{T\subseteq R_j\\ |T|=q}}
\varepsilon_j(T)
\left(\prod_{i\in T}\alpha_i\right)
\left(\prod_{i\in T'}\beta_i\right)
\det(V_T)
\det\!\left(V_{T'}^{E_\ell}\right),
\end{equation}
where
\begin{equation}\label{eq:generalized-vandermonde-beta}
\det\!\left(V_{T'}^{E_\ell}\right)
=
\prod_{\substack{r<s\\ r,s\in T'}}(x_s-x_r)\,
e_{q-\ell}\!\left((x_t)_{t\in T'}\right).
\end{equation}
\end{theorem}

\begin{proof}
We compute the minors \(M_{j,m}\). First suppose that
\(1\leq m\leq q\), so that the deleted column belongs to the
\(\alpha\)-block.  After deleting row \(j\) and column \(m\), the
remaining rows are indexed by \(R_j\). The remaining \(\alpha\)-block has exponent set \(E_m\), while the
\(\beta\)-block still has exponent set
\[
\{0,1,\dots,q-1\}.
\]

Applying the same Cauchy--Binet argument as in the proof of
Theorem~\ref{thm:det} to the matrix obtained after deleting row \(j\)
and column \(m\), a nonzero term is obtained by assigning \(q-1\)
indices to the \(\alpha\)-block and \(q\) indices to the
\(\beta\)-block. Hence the contributing subsets are precisely
\[
T\subseteq R_j,
\qquad
|T|=q-1.
\]
Then \(|T'|=q\). The indices in \(T\) contribute to the
\(\alpha\)-block, while the indices in \(T'\) contribute to the
\(\beta\)-block.

For a fixed \(T\), the diagonal weights contribute
\[
\left(\prod_{i\in T}\alpha_i\right)
\left(\prod_{i\in T'}\beta_i\right).
\]
The associated shuffle places the elements of \(T\) before the
elements of \(T'\), while preserving the natural order within each
set. By definition, the sign of this shuffle is
\[
\varepsilon_j(T)
=
(-1)^{\sum_{k=1}^{q-1}(\rho_j(t_k)-k)},
\qquad
T=\{t_1<\cdots<t_{q-1}\}.
\]
The polynomial part contributes
\[
\det\!\left(V_T^{E_m}\right)
\qquad\text{and}\qquad
\det(V_{T'}).
\]
Therefore
\[
M_{j,m}
=
\sum_{\substack{T\subseteq R_j\\ |T|=q-1}}
\varepsilon_j(T)
\left(\prod_{i\in T}\alpha_i\right)
\left(\prod_{i\in T'}\beta_i\right)
\det\!\left(V_T^{E_m}\right)
\det(V_{T'}).
\]
This proves \eqref{eq:Mjm-alpha}.

It remains to evaluate \(\det(V_T^{E_m})\). By the classical
alternant formula for Schur polynomials (see, for example,
\cite{Macdonald}), the generalized Vandermonde determinant factors as
the ordinary Vandermonde determinant times the Schur polynomial
associated with the exponent set. For \(E_m\), this Schur polynomial is
the elementary symmetric polynomial \(e_{q-m}\). Hence
\[
\det\!\left(V_T^{E_m}\right)
=
\prod_{\substack{r<s\\ r,s\in T}}(x_s-x_r)\,
e_{q-m}\!\left((x_t)_{t\in T}\right).
\]
This proves \eqref{eq:generalized-vandermonde-alpha}.

Now suppose that
\[
m=q+\ell,
\qquad
1\leq\ell\leq q.
\]
Then the deleted column belongs to the \(\beta\)-block. The
\(\alpha\)-block keeps the full exponent set
\[
\{0,1,\dots,q-1\},
\]
whereas the \(\beta\)-block has exponent set $E_\ell$.
Thus a nonzero Cauchy--Binet term must assign \(q\) indices to the
\(\alpha\)-block and \(q-1\) indices to the \(\beta\)-block. Hence the
contributing subsets satisfy
\[
T\subseteq R_j,
\qquad
|T|=q.
\]
Then \(|T'|=q-1\). The same shuffle-sign argument gives
\[
M_{j,q+\ell}
=
\sum_{\substack{T\subseteq R_j\\ |T|=q}}
\varepsilon_j(T)
\left(\prod_{i\in T}\alpha_i\right)
\left(\prod_{i\in T'}\beta_i\right)
\det(V_T)
\det\!\left(V_{T'}^{E_\ell}\right).
\]
This proves \eqref{eq:Mjm-beta}.

Similarly, the same alternant formula gives
\[
\det\!\left(V_{T'}^{E_\ell}\right)
=
\prod_{\substack{r<s\\ r,s\in T'}}(x_s-x_r)\,
e_{q-\ell}\!\left((x_t)_{t\in T'}\right).
\]
This proves \eqref{eq:generalized-vandermonde-beta} and completes the
proof.
\end{proof}
\begin{remark}
For large $q$ the $\binom{2q}{q}$ terms in
\eqref{eq:block-vandermonde-det} are too numerous for the determinant
formula to be useful for numerical computation. In practice, one may
solve the system
\[
\Lambda\mathbf{u}=\mathbf{e}_j
\]
by Gaussian elimination, which requires $O(q^3)$ operations. For
ordinary Vandermonde systems, specialized algorithms such as the
Bj\"orck--Pereyra algorithm reduce the computational complexity; see,
for example, \cite{BP}. Whether the particular block Vandermonde
structure of \eqref{eq:bv} can be exploited to obtain a comparable
fast algorithm is an interesting numerical question, which we do not
pursue here.
\end{remark}

%%%%%%%%%%%%%%%%%%%%%%%%%%%%%%%%%%%%%%%%%%%%%%%%%%%%%%%%%%%%%%%%%%%%%%%%%%%%%%%%%%%%
%%%%%%%%%%%%%%%%%%%%%%%%%%%%%%KEG-D system%%%%%%%%%%%%%%%%%%%%%%%%%%%%%%%%%%%%%%%%%%%%%%

\section{The KEG-D system with approximate boundary values}\label{sec:keg-d}

\subsection{Setting up the system}\label{ssec:setup}

Let $q\geq1$, let
\[
f\in C^q([-1,1];\mathbb R^2),
\qquad
p,r\in C^q[-1,1],
\]
and assume that $f^{(q)}$ is absolutely continuous.
For sufficiently large $N$, choose $2q$ distinct indices
$n_1,\dots,n_{2q}\in\mathbb{Z}\setminus\{0\}$ satisfying
\begin{equation}\label{eq:ns-range}
\alpha_0 N\leq |n_s|\leq N,
\qquad
s=1,\dots,2q,
\end{equation}
for some fixed constant $0<\alpha_0<1$, and form the following linear
system, obtained by discarding the remainder term $F_{n_s}$ in
\eqref{eq:cn-decomp}--\eqref{eq:PnFn}:
\begin{equation}\label{eq:keg-d-system}
c_{n_s}
=
x_{n_s}\sum_{k=0}^{q-1}x_{n_s}^{k}
\bigl[
\kappa_{n_s}^{(+)}\widetilde g_k^+
-
\kappa_{n_s}^{(-)}\widetilde g_k^-
\bigr],
\qquad
s=1,\dots,2q.
\end{equation}

Introducing the polynomial pair
\[
\widetilde G^+(x)
:=
\sum_{k=0}^{q-1}\widetilde g_k^+x^k,
\qquad
\widetilde G^-(x)
:=
\sum_{k=0}^{q-1}\widetilde g_k^-x^k,
\qquad
\deg\widetilde G^\pm\leq q-1,
\]
we can rewrite the system \eqref{eq:keg-d-system} as the interpolation
problem
\begin{equation}\label{eq:keg-d-poly}
\kappa_{n_s}^{(+)}\widetilde G^+(x_{n_s})
-
\kappa_{n_s}^{(-)}\widetilde G^-(x_{n_s})
=
\lambda_{n_s}c_{n_s},
\qquad
s=1,\dots,2q.
\end{equation}
The coefficient matrix of \eqref{eq:keg-d-poly} is
\begin{equation}\label{eq:keg-d-matrix}
\Lambda_N
:=
\begin{pmatrix}
\kappa_{n_1}^{(+)}
&
\kappa_{n_1}^{(+)}x_{n_1}
&
\cdots
&
\kappa_{n_1}^{(+)}x_{n_1}^{q-1}
&
-\kappa_{n_1}^{(-)}
&
-\kappa_{n_1}^{(-)}x_{n_1}
&
\cdots
&
-\kappa_{n_1}^{(-)}x_{n_1}^{q-1}
\\
\vdots & \vdots & & \vdots
&
\vdots & \vdots & & \vdots
\\
\kappa_{n_{2q}}^{(+)}
&
\kappa_{n_{2q}}^{(+)}x_{n_{2q}}
&
\cdots
&
\kappa_{n_{2q}}^{(+)}x_{n_{2q}}^{q-1}
&
-\kappa_{n_{2q}}^{(-)}
&
-\kappa_{n_{2q}}^{(-)}x_{n_{2q}}
&
\cdots
&
-\kappa_{n_{2q}}^{(-)}x_{n_{2q}}^{q-1}
\end{pmatrix}.
\end{equation}
Thus, $\Lambda_N$ is the block Vandermonde matrix
\eqref{eq:bv} with the identification
\begin{equation}\label{eq:identification}
\alpha_s:=\kappa_{n_s}^{(+)},
\qquad
\beta_s:=-\kappa_{n_s}^{(-)},
\qquad
x_s:=x_{n_s}.
\end{equation}

Since $\lambda_{n_s}\neq0$, multiplying the $s$-th equation of
\eqref{eq:keg-d-system} by $\lambda_{n_s}$ does not change its
solvability. Thus \eqref{eq:keg-d-system} and
\eqref{eq:keg-d-poly} are equivalent linear systems, and
$\Lambda_N$ is the coefficient matrix of the latter.
Theorem~\ref{thm:det} therefore gives an explicit criterion for the
unique solvability of \eqref{eq:keg-d-system}. By
Theorem~\ref{thm:inv-interp}, when $\Lambda_N$ is invertible, the
solution is the unique pair of polynomials $\widetilde G^\pm$ of
degree at most $q-1$ satisfying \eqref{eq:keg-d-poly}.

\subsection{Error polynomials}\label{ssec:err}

Comparing \eqref{eq:keg-d-system} with
\eqref{eq:cn-decomp}--\eqref{eq:PnFn}, we see that the error polynomials
\[
\Delta^\pm(x):=\sum_{k=0}^{q-1}\bigl(\widetilde g_k^\pm-g_k^\pm\bigr)x^k
\quad(\deg\le q-1)
\]
satisfy
\begin{equation}\label{eq:Delta-eq}
\kappa_{n_s}^{(+)}\Delta^+(x_{n_s})
-\kappa_{n_s}^{(-)}\Delta^-(x_{n_s})
=\lambda_{n_s}F_{n_s},
\qquad s=1,\dots,2q.
\end{equation}
From \eqref{eq:Delta-eq} and \eqref{eq:Fn-asymp}, setting
\begin{equation}\label{eq:Phi-def}
\Phi^\pm(x):=\Delta^\pm(x)-g_q^\pm x^q\qquad(\deg\le q),
\end{equation}
the system becomes the almost homogeneous system
\begin{equation}\label{eq:Phi-eq}
\kappa_{n_s}^{(+)}\Phi^+(x_{n_s})-\kappa_{n_s}^{(-)}\Phi^-(x_{n_s})=\varepsilon_s,
\qquad s=1,\dots,2q,
\end{equation}
where, by \eqref{eq:Fn-asymp} and \eqref{eq:ns-range},
\begin{equation}\label{eq:eps-small}
\varepsilon_s=\lambda_{n_s}\bigl(F_{n_s}
   -x_{n_s}^{q+1}[\kappa_{n_s}^{(+)}g_q^+-\kappa_{n_s}^{(-)}g_q^-]\bigr)
 =o(N^{-q})\qquad\text{uniformly in }s.
\end{equation}
The leading-order behavior of $\Delta^\pm$ is thus governed by how
small the left-hand side of \eqref{eq:Phi-eq} can be, and this depends
on the parity of the indices $n_s$, which we study in the next
subsection.

\subsection{Parity classification and kernel basis}\label{ssec:parity}

By \eqref{eq:kappapm-asymp},
\begin{equation}\label{eq:kappa-leading}
\kappa_n^{(+)}
=
(-1)^nK_0+O\!\left(\frac{1}{n}\right),
\qquad
-\kappa_n^{(-)}
=
-K_0+O\!\left(\frac{1}{n}\right),
\end{equation}
where \(K_0=1/\sqrt{2}\). We partition the selected indices according
to parity:
\[
E:=\{s\in[2q]:n_s\text{ is even}\},
\qquad
O:=\{s\in[2q]:n_s\text{ is odd}\},
\]
and set
\[
q_e:=|E|,
\qquad
q_o:=|O|,
\qquad
q_e+q_o=2q.
\]

We now examine the leading-order structure of the KEG-D coefficient
matrix \(\Lambda_N\). In view of \eqref{eq:kappa-leading}, define
\begin{equation}\label{eq:Lambda0}
\Lambda_0
:=
K_0
\begin{pmatrix}
(-1)^{n_1}
&
(-1)^{n_1}x_{n_1}
&
\cdots
&
(-1)^{n_1}x_{n_1}^{q-1}
&
-1
&
-x_{n_1}
&
\cdots
&
-x_{n_1}^{q-1}
\\
\vdots & \vdots & & \vdots
&
\vdots & \vdots & & \vdots
\\
(-1)^{n_{2q}}
&
(-1)^{n_{2q}}x_{n_{2q}}
&
\cdots
&
(-1)^{n_{2q}}x_{n_{2q}}^{q-1}
&
-1
&
-x_{n_{2q}}
&
\cdots
&
-x_{n_{2q}}^{q-1}
\end{pmatrix}.
\end{equation}
Thus, \(\Lambda_0\) is obtained from \(\Lambda_N\) by replacing
\(\kappa_{n_s}^{(+)}\) and \(-\kappa_{n_s}^{(-)}\) by their leading
terms
\[
(-1)^{n_s}K_0
\qquad\text{and}\qquad
-K_0,
\]
respectively. The notation \(\Lambda_0\) refers to this leading-order
replacement; the matrix still depends on \(N\) through the nodes
\(x_{n_s}\).

We determine which parity distribution makes \(\Lambda_0\)
non-singular. For each exponent \(j=0,\dots,q-1\), consider the
corresponding pair of columns
\[
K_0\bigl((-1)^{n_s}x_{n_s}^j\bigr)_{s=1}^{2q},
\qquad
-K_0\bigl(x_{n_s}^j\bigr)_{s=1}^{2q},
\]
and replace them by their sum and difference. The determinant of each
such two-column transformation is \(-2\), so the rank is unchanged.

For the sum, the \(s\)-th component is
\[
K_0\bigl((-1)^{n_s}-1\bigr)x_{n_s}^j
=
\begin{cases}
0, & s\in E,\\
-2K_0x_{n_s}^j, & s\in O,
\end{cases}
\]
whereas for the difference it is
\[
K_0\bigl((-1)^{n_s}+1\bigr)x_{n_s}^j
=
\begin{cases}
2K_0x_{n_s}^j, & s\in E,\\
0, & s\in O.
\end{cases}
\]
Hence, after a permutation of rows and columns, the transformed matrix
has the block-diagonal form
\[
\begin{pmatrix}
-2K_0V_O & 0\\
0 & 2K_0V_E
\end{pmatrix},
\]
where \(V_O\) is the \(q_o\times q\) Vandermonde matrix on the nodes
\(x_{n_s}\), \(s\in O\), and \(V_E\) is the \(q_e\times q\)
Vandermonde matrix on the nodes \(x_{n_s}\), \(s\in E\).

For \(\Lambda_0\) to have full rank \(2q\), both \(V_O\) and \(V_E\)
must have rank \(q\). Hence necessarily
\[
q_o\geq q,
\qquad
q_e\geq q.
\]
Since
\[
q_e+q_o=2q,
\]
these two inequalities hold simultaneously if and only if
\begin{equation}\label{eq:parity-balance}
q_e=q_o=q.
\end{equation}

Conversely, under \eqref{eq:parity-balance}, both \(V_E\) and \(V_O\)
are ordinary \(q\times q\) Vandermonde matrices. Since the indices
\(n_1,\dots,n_{2q}\) are distinct, the eigenvalues
\(\lambda_{n_1},\dots,\lambda_{n_{2q}}\), and hence the nodes
\(x_{n_1},\dots,x_{n_{2q}}\), are pairwise distinct. Therefore both
\(V_E\) and \(V_O\) are non-singular. Since the row and column
transformations used above are invertible, it follows that
\[
\det\Lambda_0\neq0.
\]

Thus, the strict parity balance \eqref{eq:parity-balance} is necessary
and sufficient for the leading-order matrix \(\Lambda_0\) to be
non-singular. This conclusion concerns \(\Lambda_0\) only; the
invertibility of the actual KEG-D coefficient matrix \(\Lambda_N\)
requires a separate comparison with \(\Lambda_0\). In particular,
\eqref{eq:parity-balance} explains the choice of exactly \(q\) even
and \(q\) odd indices.

We now study the leading-order structure of the error system
\eqref{eq:Phi-eq}. Replacing
\(\kappa_{n_s}^{(+)}\) and \(\kappa_{n_s}^{(-)}\) by their leading
terms from \eqref{eq:kappapm-asymp} and dividing by \(K_0\), we are
led to the linear evaluation map
\[
\mathcal E:
\mathcal P_q\times\mathcal P_q
\longrightarrow
\mathbb C^{2q},
\]
defined by
\[
\mathcal E(\Phi^+,\Phi^-)
:=
\bigl(
(-1)^{n_s}\Phi^+(x_{n_s})
-
\Phi^-(x_{n_s})
\bigr)_{s=1}^{2q},
\]
where \(\mathcal P_q\) denotes the space of polynomials of degree at
most \(q\).

A pair \((\Phi^+,\Phi^-)\) belongs to \(\ker\mathcal E\) if and only
if
\begin{equation}\label{eq:WW'}
\begin{aligned}
W_E(x_{n_s})&=0, && s\in E,\\
W_O(x_{n_s})&=0, && s\in O,
\end{aligned}
\end{equation}
where
\begin{equation}\label{eq:WW'-def}
W_E(x):=\Phi^+(x)-\Phi^-(x),
\qquad
W_O(x):=-\Phi^+(x)-\Phi^-(x).
\end{equation}
Both \(W_E\) and \(W_O\) have degree at most \(q\). We also define
\begin{equation}\label{eq:Pi-EO}
\Pi_E(x):=\prod_{s\in E}(x-x_{n_s}),
\qquad
\Pi_O(x):=\prod_{s\in O}(x-x_{n_s}).
\end{equation}

\begin{lemma}[Kernel dimension]\label{lem:kerdim}
The kernel of \(\mathcal E\) has dimension \(2\) if and only if
\[
q_e\leq q+1,
\qquad
q_o\leq q+1.
\]
\end{lemma}

\begin{proof}
The correspondence
\[
(\Phi^+,\Phi^-)
\longleftrightarrow
(W_E,W_O)
\]
defined by \eqref{eq:WW'-def} is a linear bijection of
\(\mathcal P_q\times\mathcal P_q\) onto itself.

For a pair \((\Phi^+,\Phi^-)\in\ker\mathcal E\), the corresponding
polynomial \(W_E\in\mathcal P_q\) must vanish at the \(q_e\) distinct
nodes \(x_{n_s}\), \(s\in E\).
Therefore, the space of possible \(W_E\) has dimension
\[
\max\{0,q+1-q_e\}.
\]
Similarly, the space of possible \(W_O\) has dimension
\[
\max\{0,q+1-q_o\}.
\]
Hence
\[
\dim\ker\mathcal E
=
\max\{0,q+1-q_e\}
+
\max\{0,q+1-q_o\}.
\]

If
\[
q_e\leq q+1,
\qquad
q_o\leq q+1,
\]
then
\[
\dim\ker\mathcal E
=
(q+1-q_e)+(q+1-q_o)
=
2,
\]
because \(q_e+q_o=2q\).

If, for example, \(q_e\geq q+2\), then the first term vanishes and
\[
\dim\ker\mathcal E
=
q+1-q_o.
\]
Since \(q_o=2q-q_e\),
\[
q+1-q_o
=
q_e-q+1
\geq3.
\]
The case \(q_o\geq q+2\) is analogous. This proves the claim.
\end{proof}

Under the strict-balance condition \eqref{eq:parity-balance},
\[
|E|=|O|=q.
\]
Hence, for every \((\Phi^+,\Phi^-)\in\ker\mathcal E\),
\[
W_E=c_E\Pi_E,
\qquad
W_O=c_O\Pi_O,
\]
for some constants \(c_E,c_O\). Since \(\Pi_E\) and \(\Pi_O\) are
monic of degree \(q\), if
\[
\Phi^\pm(x)=\sum_{k=0}^q\Phi_k^\pm x^k,
\]
then by \eqref{eq:WW'-def},
\[
c_E=\Phi_q^+-\Phi_q^-,
\qquad
c_O=-\Phi_q^+-\Phi_q^-.
\]
Thus, if
\[
(\Phi_q^+,\Phi_q^-)=(g^+,g^-),
\]
then
\[
W_E=(g^+-g^-)\Pi_E,
\qquad
W_O=(-g^+-g^-)\Pi_O.
\]
Conversely, these formulas determine \(\Phi^\pm\) uniquely through
\[
\Phi^+=\frac{W_E-W_O}{2},
\qquad
\Phi^-=-\frac{W_E+W_O}{2}.
\]
Therefore, the map
\[
(\Phi^+,\Phi^-)\longmapsto(\Phi_q^+,\Phi_q^-)
\]
is a linear bijection from \(\ker\mathcal E\) onto \(\mathbb C^2\).

We may therefore define \((U^+,U^-)\) to be the unique element of
\(\ker\mathcal E\) satisfying
\[
U_q^+=1,
\qquad
U_q^-=0,
\]
and \((V^+,V^-)\) to be the unique element satisfying
\[
V_q^+=0,
\qquad
V_q^-=1,
\]
where \(U_q^\pm\) and \(V_q^\pm\) denote the coefficients of \(x^q\)
in \(U^\pm\) and \(V^\pm\), respectively.

By \eqref{eq:Phi-def},
\[
\Phi^\pm(x)
=
\Delta^\pm(x)-g_q^\pm x^q,
\]
and since \(\deg\Delta^\pm\leq q-1\), the coefficients of \(x^q\) in
\(\Phi^\pm\) are
\[
\Phi_q^+=-g_q^+,
\qquad
\Phi_q^-=-g_q^-.
\]
Hence the kernel element having the same leading coefficients as the
error pair \((\Phi^+,\Phi^-)\) is
\[
-g_q^+(U^+,U^-)-g_q^-(V^+,V^-).
\]
This suggests the approximation
\[
\Phi^\pm(x)
\approx
-g_q^+U^\pm(x)-g_q^-V^\pm(x),
\]
which is made precise in Lemma~\ref{lem:stab} and
Theorem~\ref{thm:gjumps} below.

The marginal distributions
\[
(q_e,q_o)=(q+1,q-1)
\qquad\text{or}\qquad
(q_e,q_o)=(q-1,q+1)
\]
also give
\[
\dim\ker\mathcal E=2,
\]
but they do not satisfy the leading-order non-singularity condition
\eqref{eq:parity-balance}. They will therefore be considered
separately in Case~B below.

%%%%%%%%%%%%%%%%%%%%%%%%%%%%%%%%%%%%%%%%%%%%%%%%%%%%%%%%%%%%%%%%%%%%%%%%%%

\subsubsection*{Case A: strict balance \(q_e=q_o=q\)}

\begin{lemma}\label{lem:UV-balanced}
Under the strict-balance condition \(q_e=q_o=q\),
\begin{equation}\label{eq:UV-explicit}
\begin{aligned}
U^+(x)
&=
\frac{1}{2}\bigl[\Pi_E(x)+\Pi_O(x)\bigr],
&
U^-(x)
&=
-\frac{1}{2}\bigl[\Pi_E(x)-\Pi_O(x)\bigr],
\\[2pt]
V^+(x)
&=
-\frac{1}{2}\bigl[\Pi_E(x)-\Pi_O(x)\bigr],
&
V^-(x)
&=
\frac{1}{2}\bigl[\Pi_E(x)+\Pi_O(x)\bigr].
\end{aligned}
\end{equation}
\end{lemma}

\begin{proof}
For a kernel element with coefficients of \(x^q\) given by
\((g^+,g^-)\), the preceding characterization of
\(\ker\mathcal E\) gives
\[
W_E=(g^+-g^-)\Pi_E,
\qquad
W_O=(-g^+-g^-)\Pi_O.
\]
For \((U^+,U^-)\), we have \((g^+,g^-)=(1,0)\), and hence
\[
W_E=\Pi_E,
\qquad
W_O=-\Pi_O.
\]
For \((V^+,V^-)\), we have \((g^+,g^-)=(0,1)\), and hence
\[
W_E=-\Pi_E,
\qquad
W_O=-\Pi_O.
\]
Using \eqref{eq:WW'-def}, equivalently
\[
\Phi^+=\frac{W_E-W_O}{2},
\qquad
\Phi^-=-\frac{W_E+W_O}{2},
\]
gives \eqref{eq:UV-explicit}.
\end{proof}

\subsubsection*{Case B: marginal imbalance \(q_e=q+1\) or \(q_o=q+1\)}

Suppose first that
\[
q_e=q+1,
\qquad
q_o=q-1.
\]
Since \(W_E\in\mathcal P_q\) vanishes at \(q+1\) distinct nodes,
\(W_E\equiv0\). Hence, by \eqref{eq:WW'-def},
\(\Phi^+=\Phi^-\), and therefore
\[
\Phi_q^+=\Phi_q^-.
\]
Using \eqref{eq:Phi-def}, this gives the compatibility condition
\[
g_q^+=g_q^-,
\]
which need not hold for a general \(f\). Thus the leading-order system
is degenerate; equivalently, \(\Lambda_0\) is singular.

Any invertibility of the full KEG-D matrix must therefore come from the
\(O(1/n)\)-corrections in \eqref{eq:kappa-leading}. For \(p=r=0\),
these corrections vanish, so the coefficient matrix remains singular
for every \(N\).

The case \(q_o=q+1\), \(q_e=q-1\) is analogous. Hence the subsequent
analysis is restricted to the strict-balance condition \(q_e=q_o=q\).

Finally, for odd \(q\), the symmetric choice
\[
\pm N,\ \pm(N-1),\dots
\]
falls into the marginal case: among the \(2q\) selected indices,
\(q+1\) have the same parity as \(N\), while \(q-1\) have the opposite
parity.

\subsection{Asymptotic estimates for $\widetilde g_k^\pm-g_k^\pm$}

We now restrict to the strictly balanced case and assume, in addition,
that the limiting nodes are distinct within each parity class. Under
these assumptions, the following lemma shows that a solution of
\eqref{eq:Phi-eq} remains quantitatively close to the unique element of
$\ker\mathcal E$ whose coefficients of $x^q$ are $-g_q^+$ and
$-g_q^-$.

\begin{lemma}[Stability]\label{lem:stab}
Assume $q_e=q_o=q$, and let the indices satisfy \eqref{eq:ns-range}
with
\[
\lim_{N\to\infty}\frac{n_s}{N}=c_s\neq0,
\qquad
s=1,\dots,2q.
\]
Assume that the numbers $\{c_s:s\in E\}$ are pairwise distinct, as are
$\{c_s:s\in O\}$. Let
\[
\Phi^\pm(x)=\sum_{k=0}^q\Phi_k^\pm x^k
\]
satisfy \eqref{eq:Phi-eq} with
\[
\Phi_q^\pm=-g_q^\pm,
\]
and put
\[
\varepsilon:=\max_s|\varepsilon_s|.
\]
Then, for all sufficiently large $N$,
\begin{equation}\label{eq:stab-bound}
\Phi_k^\pm+g_q^+U_k^\pm+g_q^-V_k^\pm
=
O\!\left(N^k\varepsilon+N^{k-q-1}\right),
\qquad
k=0,\dots,q-1,
\end{equation}
where $U_k^\pm$ and $V_k^\pm$ denote the coefficients of $x^k$ in
the polynomials $U^\pm$ and $V^\pm$ given by
\eqref{eq:UV-explicit}. Moreover, the matrix $\Lambda_N$ is
invertible for all sufficiently large $N$.
\end{lemma}

\begin{proof}
Define
\begin{equation}\label{eq:RE-RO-def}
R_E
:=
\Phi^+-\Phi^-+(g_q^+-g_q^-)\Pi_E,
\qquad
R_O
:=
-\Phi^+-\Phi^--(g_q^++g_q^-)\Pi_O.
\end{equation}
where $\Pi_E$ and $\Pi_O$ are given by \eqref{eq:Pi-EO}. Since
\[
\Phi_q^+=-g_q^+,
\qquad
\Phi_q^-=-g_q^-,
\]
and $\Pi_E,\Pi_O$ are monic polynomials of degree $q$, we have
\[
\deg R_E,\deg R_O\leq q-1.
\]

Using \eqref{eq:UV-explicit}, we obtain
\begin{equation}\label{eq:R-comb}
\begin{aligned}
\Phi^++g_q^+U^++g_q^-V^+
&=
\frac{1}{2}(R_E-R_O),
\\
\Phi^-+g_q^+U^-+g_q^-V^-
&=
-\frac{1}{2}(R_E+R_O).
\end{aligned}
\end{equation}
Thus it suffices to estimate the coefficients of $R_E$ and $R_O$.

By \eqref{eq:kappa-leading} and \eqref{eq:ns-range}, we may write
\[
\kappa_{n_s}^{(+)}
=
(-1)^{n_s}K_0+\delta_s^+,
\qquad
-\kappa_{n_s}^{(-)}
=
-K_0+\delta_s^-,
\]
where
\[
\delta_s^\pm=O(N^{-1})
\]
uniformly in $s$. Substituting these expressions into
\eqref{eq:Phi-eq}, and using
$\Pi_E(x_{n_s})=0$ for $s\in E$ and
$\Pi_O(x_{n_s})=0$ for $s\in O$, gives
\begin{equation}\label{eq:R-values}
\begin{aligned}
K_0R_E(x_{n_s})
&=
\varepsilon_s
-\delta_s^+\Phi^+(x_{n_s})
-\delta_s^-\Phi^-(x_{n_s}),
&& s\in E,
\\
K_0R_O(x_{n_s})
&=
\varepsilon_s
-\delta_s^+\Phi^+(x_{n_s})
-\delta_s^-\Phi^-(x_{n_s}),
&& s\in O.
\end{aligned}
\end{equation}

Set
\begin{equation}\label{eq:m-def}
m:=
\max\left\{
\max_{s\in E}|R_E(x_{n_s})|,
\max_{s\in O}|R_O(x_{n_s})|
\right\}.
\end{equation}

By \eqref{eq:lam-asymp},
\[
Nx_{n_s}
=
\frac{N}{\lambda_{n_s}}
\longrightarrow
\frac{2}{\pi c_s}.
\]
Hence there exist constants $C,d>0$ such that, for all sufficiently
large $N$,
\begin{equation}\label{eq:node-bounds}
|x_{n_s}|\leq\frac{C}{N},
\qquad
|x_{n_s}-x_{n_t}|\geq\frac{d}{N}
\end{equation}
whenever $s\neq t$ belong to the same parity class.

Let $\ell_t^E$, $t\in E$, be the Lagrange basis polynomials associated
with the $E$-nodes. Writing
\[
\ell_t^E(x)
=
\sum_{k=0}^{q-1}\ell_{t,k}^E x^k,
\]
the node estimates \eqref{eq:node-bounds} imply
\begin{equation}\label{eq:lagrange-bounds}
|\ell_t^E(y)|\leq C_1,
\qquad
|y|\leq\frac{C}{N},
\qquad
\ell_{t,k}^E=O(N^k),
\quad
k=0,\dots,q-1.
\end{equation}
The same estimates hold for the Lagrange basis polynomials
$\ell_t^O$, $t\in O$.

Since $\deg R_E,\deg R_O\leq q-1$, Lagrange interpolation gives
\[
R_E(x)
=
\sum_{t\in E}R_E(x_{n_t})\ell_t^E(x),
\qquad
R_O(x)
=
\sum_{t\in O}R_O(x_{n_t})\ell_t^O(x).
\]
By \eqref{eq:m-def} and \eqref{eq:lagrange-bounds},
\[
|R_E(y)|+|R_O(y)|
\leq
C_2m,
\qquad
|y|\leq\frac{C}{N}.
\]
If
\[
R_E(x)=\sum_{k=0}^{q-1}R_{E,k}x^k,
\qquad
R_O(x)=\sum_{k=0}^{q-1}R_{O,k}x^k,
\]
then \eqref{eq:lagrange-bounds} also gives
\begin{equation}\label{eq:R-coeff-bound}
R_{E,k},R_{O,k}
=
O(mN^k),
\qquad
k=0,\dots,q-1.
\end{equation}

For every selected node $y=x_{n_s}$, the definitions of $R_E$ and
$R_O$ give
\[
\begin{aligned}
\Phi^+(y)
&=
\frac{1}{2}\Bigl[
R_E(y)-R_O(y)
-(g_q^+-g_q^-)\Pi_E(y)
-(g_q^++g_q^-)\Pi_O(y)
\Bigr],
\\
\Phi^-(y)
&=
-\frac{1}{2}\Bigl[
R_E(y)+R_O(y)
-(g_q^+-g_q^-)\Pi_E(y)
+(g_q^++g_q^-)\Pi_O(y)
\Bigr].
\end{aligned}
\]
By \eqref{eq:node-bounds} and \eqref{eq:Pi-EO},
\[
|\Pi_E(y)|,\ |\Pi_O(y)|=O(N^{-q}),
\]
and therefore
\begin{equation}\label{eq:Phi-node-bound}
|\Phi^\pm(y)|
\leq
C_3\bigl(N^{-q}+m\bigr),
\qquad
y=x_{n_s}.
\end{equation}

Combining \eqref{eq:R-values} with \eqref{eq:Phi-node-bound} and using
$\delta_s^\pm=O(N^{-1})$, we obtain
\begin{equation}\label{eq:m-ineq}
K_0m
\leq
\varepsilon
+
C_4N^{-1}\bigl(N^{-q}+m\bigr).
\end{equation}
Hence, for all sufficiently large $N$,
\begin{equation}\label{eq:m-bound}
m
\leq
C_5\bigl(\varepsilon+N^{-q-1}\bigr).
\end{equation}
Combining \eqref{eq:m-bound} with \eqref{eq:R-coeff-bound}, we obtain
\begin{equation}\label{eq:R-coeff-final}
R_{E,k},R_{O,k}
=
O\!\left(N^k\varepsilon+N^{k-q-1}\right),
\qquad
k=0,\dots,q-1.
\end{equation}
Combining \eqref{eq:R-coeff-final} with \eqref{eq:R-comb} yields
\eqref{eq:stab-bound}.

It remains to prove that $\Lambda_N$ is invertible for all sufficiently
large $N$. 
Suppose that the homogeneous version of \eqref{eq:Delta-eq} has a
solution
\[
\Delta^\pm\in\mathcal P_{q-1}.
\]
Then \eqref{eq:Phi-eq} holds with
\[
\varepsilon_s=0,
\qquad
g_q^\pm=0.
\]
Hence, by
\eqref{eq:RE-RO-def},
\[
R_E=\Delta^+-\Delta^-,
\qquad
R_O=-\Delta^+-\Delta^-.
\]
Since $\varepsilon=0$ and $g_q^\pm=0$, the estimate
\eqref{eq:Phi-node-bound} reduces to
\[
|\Phi^\pm(y)|\leq C_3m.
\]
Substituting this into \eqref{eq:R-values} gives
\[
m\leq \frac{C_4}{K_0N}\,m.
\]
For sufficiently large $N$, this forces $m=0$. Thus $R_E$ and $R_O$
vanish at the $q$ nodes of their respective parity classes. Since
\[
\deg R_E,\deg R_O\leq q-1,
\]
we have
\[
R_E\equiv R_O\equiv0.
\]
Hence
\[
\Delta^+=\Delta^-=0.\]
Thus the homogeneous system has only the trivial solution, and
$\Lambda_N$ is invertible for all sufficiently large $N$.
\end{proof}

Combining Lemma~\ref{lem:stab} with \eqref{eq:eps-small}, we obtain
the decomposition
\begin{equation}\label{eq:Phi-decomp-2}
\Phi^\pm(x)
=
-g_q^+U^\pm(x)-g_q^-V^\pm(x)+\rho^\pm(x),
\end{equation}
where
\[
\rho^\pm(x)
=
\sum_{k=0}^{q-1}\rho_k^\pm x^k,
\qquad
\rho_k^\pm=o(N^{k-q}),
\qquad
k=0,\dots,q-1.
\]
In particular, the coefficient of $x^q$ in $\rho^\pm$ is exactly
zero.

\begin{theorem}\label{thm:gjumps}
Suppose
\[
f\in C^q([-1,1];\mathbb R^2),
\qquad
p,r\in C^q[-1,1],
\]
with $f^{(q)}$ absolutely continuous, where $q\geq1$. Assume that the
indices $n_s=n_s(N)$ satisfy \eqref{eq:ns-range} with
\[
\lim_{N\to\infty}\frac{n_s}{N}=c_s\neq0,
\qquad
s=1,\dots,2q,
\]
that the strict-balance condition \eqref{eq:parity-balance} holds, and
that the limits $\{c_s:s\in E\}$ are pairwise distinct, as are
$\{c_s:s\in O\}$. Then
\begin{equation}\label{eq:tilde-g-asymp}
\widetilde g_k^\pm
=
g_k^\pm
-g_q^+U_k^\pm
-g_q^-V_k^\pm
+o(N^{k-q}),
\qquad
k=0,\dots,q-1,
\end{equation}
where $U_k^\pm$ and $V_k^\pm$ denote the coefficients of $x^k$ in
the polynomials $U^\pm$ and $V^\pm$ given by
\eqref{eq:UV-explicit}. In particular,
\[
\widetilde g_k^\pm-g_k^\pm
=
O(N^{k-q}),
\qquad
k=0,\dots,q-1.
\]
\end{theorem}

\begin{proof}
By Lemma~\ref{lem:stab}, the matrix $\Lambda_N$ is invertible for all
sufficiently large $N$, so the KEG-D system has a unique solution.
Moreover, by \eqref{eq:eps-small},
\[
\varepsilon=\max_s|\varepsilon_s|=o(N^{-q}).
\]
Hence \eqref{eq:Phi-decomp-2} gives
\[
\Phi^\pm(x)
=
-g_q^+U^\pm(x)-g_q^-V^\pm(x)+\rho^\pm(x),
\]
where
\[
\rho^\pm(x)
=
\sum_{k=0}^{q-1}\rho_k^\pm x^k,
\qquad
\rho_k^\pm=o(N^{k-q}).
\]
By \eqref{eq:Phi-def}, for $k=0,\dots,q-1$,
\[
\Phi_k^\pm
=
\widetilde g_k^\pm-g_k^\pm.
\]
Comparing the coefficients of $x^k$ therefore gives
\eqref{eq:tilde-g-asymp}.

Finally, by \eqref{eq:UV-explicit} and the node estimates
\eqref{eq:node-bounds},
\[
U_k^\pm,V_k^\pm=O(N^{k-q}),
\qquad
k=0,\dots,q-1.
\]
The final assertion follows.
\end{proof}
\begin{remark}\label{rem:mult}
The distinctness assumption on the limits within each parity class
excludes configurations in which several of the ratios $n_s/N$
converge to the same limit within $E$ or within $O$. This is analogous
to the trigonometric case considered in \cite{BBP}, where a
multiplicity $\alpha$ among the limiting values requires
\[
f\in C^{q+\alpha-1}
\]
with $f^{(q+\alpha-1)}$ absolutely continuous.

One expects that \eqref{eq:tilde-g-asymp} and the corresponding
$L_2$-error asymptotics remain valid for clustered configurations when
the maximal multiplicity $\alpha$ is counted separately within each
parity class, provided
\[
f\in C^{q+\alpha-1}([-1,1];\mathbb R^2),
\]
with $f^{(q+\alpha-1)}$ absolutely continuous, and $p,r$ are
sufficiently smooth for the asymptotic expansions in
\eqref{eq:kappapm-asymp} to be continued to the required order in
$1/n$.

We do not prove this extension here. As in
\cite[Theorem~3.3]{BBP}, further integrations by parts extending
\eqref{eq:Fn-asymp} would introduce the additional reduced boundary
values
\[
g_{q+1}^\pm,\dots,g_{q+\alpha-1}^\pm,
\]
which would enter the interpolation equations through higher powers of
$x_{n_s}$.
\end{remark}

%%%%%%%%%%%%%%%%%%%%%%%%%%%%%%%%%%%%%%%%%%%%%%%%%%%%%%%%%%%%%%%%%%%

\section{Asymptotic $L_2$ error estimates}\label{sec:l2}
As in \eqref{eq:Sqn-exact}, define
\begin{equation}\label{eq:Sqn-tilde}
\widetilde S_{q,N}(f)
:=
\widetilde P(x)
+
\sum_{|n|\leq N}\widetilde F_n v_n(x),
\end{equation}
where
\[
\widetilde P(x)
:=
\sum_{k=0}^{q-1}
\bigl[
G_k(x,1,0)\widetilde f_k(1)
-
G_k(x,-1,0)\widetilde f_k(-1)
\bigr].
\]
For the reduced boundary values $\widetilde g_k^\pm$ obtained from
\eqref{eq:keg-d-system}, we choose
\[
\widetilde f_k(\pm1)
:=
\widetilde g_k^\pm
\binom{\cos\gamma_\pm}{-\sin\gamma_\pm},
\qquad
\gamma_+=\beta,
\qquad
\gamma_-=\alpha,
\]
so that, consistently with \eqref{eq:gk},
\[
\begin{pmatrix}
\cos\gamma_\pm & -\sin\gamma_\pm
\end{pmatrix}
\widetilde f_k(\pm1)
=
\widetilde g_k^\pm.
\]
We then set
\[
\widetilde P_n
:=
\langle \widetilde P,v_n\rangle_{L_2},
\qquad
\widetilde F_n
:=
c_n-\widetilde P_n.
\]
The \emph{KEG-D error} is defined by
\[
\widetilde R_{q,N}(f)
:=
f-\widetilde S_{q,N}(f).
\]

\subsection{Main theorem}

\begin{theorem}\label{thm:L2}
Under the hypotheses of Theorem~\ref{thm:gjumps} (in particular, the
strict parity balance $q_e=q_o=q$),
\begin{equation}\label{eq:L2}
\lim_{N\to\infty}N^{2q+1}\|\widetilde R_{q,N}(f)\|^2
=\frac{2^{2q}}{\pi^{2q+2}}\bigl[(A+B)^2\,J_E+(A-B)^2\,J_O\bigr],
\end{equation}
where $A,B$ are the Dirac jump constants \eqref{eq:AB} and
\begin{equation}\label{eq:JEO}
J_E:=\int_{-1}^{1}\!\prod_{s\in E}\!\Bigl(x-\frac{1}{c_s}\Bigr)^{\!2}dx,
\qquad
J_O:=\int_{-1}^{1}\!\prod_{s\in O}\!\Bigl(x-\frac{1}{c_s}\Bigr)^{\!2}dx.
\end{equation}
\end{theorem}

\begin{proof} 
Since $\widetilde F_n$ is the generalized Fourier coefficient of
$f-\widetilde P$, we have
\[
\widetilde R_{q,N}(f)
=
\sum_{|n|>N}\widetilde F_n v_n.
\]
Hence, by Parseval,
\[
\|\widetilde R_{q,N}(f)\|^2
=
\sum_{|n|>N}|\widetilde F_n|^2.
\]
Since
\[
\widetilde F_n
=
F_n+(P_n-\widetilde P_n),
\]
we first compute the second term. By the definitions of $P_n$ and
$\widetilde P_n$,
\begin{align*}
P_n-\widetilde P_n
&=
-x_n\bigl[
\kappa_n^{(+)}\Delta^+(x_n)
-\kappa_n^{(-)}\Delta^-(x_n)
\bigr]
\\
&=
-x_n\bigl[
\kappa_n^{(+)}\Phi^+(x_n)
-\kappa_n^{(-)}\Phi^-(x_n)
\bigr]
\\
&\quad
-x_n^{q+1}
\bigl[
\kappa_n^{(+)}g_q^+
-\kappa_n^{(-)}g_q^-
\bigr],
\end{align*}
where we used \eqref{eq:Phi-def}. Define
\[
\mathcal Q_n
:=
\kappa_n^{(+)}\Phi^+(x_n)
-\kappa_n^{(-)}\Phi^-(x_n).
\]
Then \eqref{eq:Fn-asymp} gives
\[
P_n-\widetilde P_n
=
-x_n\mathcal Q_n-F_n+r_{n},
\qquad
|n|>N,
\]
where
\[
r_{n}=o(|n|^{-q-1})
\]
uniformly for $|n|>N$. Therefore
\begin{equation}\label{eq:tildeFn-id}
\widetilde F_n
=
-x_n\mathcal Q_n+r_{n}.
\end{equation}

By \eqref{eq:Phi-decomp-2},
\begin{align*}
\mathcal Q_n
&=
-g_q^+
\bigl[
\kappa_n^{(+)}U^+(x_n)
-\kappa_n^{(-)}U^-(x_n)
\bigr]
\\
&\quad
-g_q^-
\bigl[
\kappa_n^{(+)}V^+(x_n)
-\kappa_n^{(-)}V^-(x_n)
\bigr]
\\
&\quad
+\kappa_n^{(+)}\rho^+(x_n)
-\kappa_n^{(-)}\rho^-(x_n).
\end{align*}
Since
\[
\rho^\pm(x)
=
\sum_{k=0}^{q-1}\rho_k^\pm x^k,
\qquad
\rho_k^\pm=o(N^{k-q}),
\]
and $|x_n|=O(N^{-1})$ uniformly for $|n|>N$, it follows that
\[
\rho^\pm(x_n)=o(N^{-q}),
\]
uniformly for $|n|>N$ as $N\to\infty$.

Using \eqref{eq:kappa-leading} and \eqref{eq:UV-explicit}, we now
distinguish the parity of the tail index $n$. For $|n|>N$, we have
\[
U^\pm(x_n),V^\pm(x_n)=O(N^{-q}).
\]
Hence, if $n$ is even,
\[
\kappa_n^{(+)}U^+(x_n)-\kappa_n^{(-)}U^-(x_n)
=
K_0\Pi_E(x_n)
+
o\!\left(N^{-q}\right),
\]
and
\[
\kappa_n^{(+)}V^+(x_n)-\kappa_n^{(-)}V^-(x_n)
=
-K_0\Pi_E(x_n)
+
o\!\left(N^{-q}\right).
\]
If $n$ is odd,
\[
\kappa_n^{(+)}U^+(x_n)-\kappa_n^{(-)}U^-(x_n)
=
-K_0\Pi_O(x_n)
+
o\!\left(N^{-q}\right),
\]
and
\[
\kappa_n^{(+)}V^+(x_n)-\kappa_n^{(-)}V^-(x_n)
=
-K_0\Pi_O(x_n)
+
o\!\left(N^{-q}\right).
\]
Hence
\begin{equation}\label{eq:Qn-cases}
\mathcal Q_n
=
\begin{cases}
-K_0(g_q^+-g_q^-)\Pi_E(x_n)+o\!\left(N^{-q}\right),
& n\text{ even},
\\[2mm]
\phantom{-}K_0(g_q^++g_q^-)\Pi_O(x_n)+o\!\left(N^{-q}\right),
& n\text{ odd}.
\end{cases}
\end{equation}

Combining \eqref{eq:tildeFn-id} and \eqref{eq:Qn-cases}, and using
\eqref{eq:Pi-EO}, we obtain
\begin{equation}\label{eq:tildeFn-prod}
\widetilde F_n
=
\begin{cases}
\displaystyle
K_0(g_q^+-g_q^-)\,
x_n^{q+1}
\prod_{s\in E}
\left(1-\frac{n}{n_s}\right)
+x_n\sigma_{n,N},
& n\text{ even},
\\[4mm]
\displaystyle
-K_0(g_q^++g_q^-)\,
x_n^{q+1}
\prod_{s\in O}
\left(1-\frac{n}{n_s}\right)
+x_n\sigma_{n,N},
& n\text{ odd},
\end{cases}
\end{equation}
where
\[
\sup_{|n|>N}|\sigma_{n,N}|=o(N^{-q}).
\]
Here we used \eqref{eq:Pi-EO} and \eqref{eq:lam-asymp}. Indeed,
\[
nx_n=\frac{2}{\pi}+O\!\left(\frac{1}{n}\right),
\]
and hence, uniformly for $|n|>N$,
\[
(x_n-x_{n_s})
-
x_n\left(1-\frac{n}{n_s}\right)
=
\frac{nx_n-n_sx_{n_s}}{n_s}
=
O(N^{-2}).
\]
It follows that
\[
x_n\Pi_E(x_n)
=
x_n^{q+1}
\prod_{s\in E}
\left(1-\frac{n}{n_s}\right)
+
x_nO(N^{-q-1}),
\]
and analogously for $\Pi_O$.

The remainders in \eqref{eq:tildeFn-prod} are of the form
$x_n\sigma_{n,N}$, where
\[
\sup_{|n|>N}|\sigma_{n,N}|=o(N^{-q}).
\]
Since
\[
\sum_{|n|>N}|x_n|^2=O(N^{-1}),
\]
we have
\[
\sum_{|n|>N}|x_n\sigma_{n,N}|^2
=
o(N^{-2q-1}).
\]
The squared sum of the leading terms in \eqref{eq:tildeFn-prod} is
$O(N^{-2q-1})$, so the corresponding cross terms are also
$o(N^{-2q-1})$ by the Cauchy--Schwarz inequality. Thus only the
leading terms in \eqref{eq:tildeFn-prod} contribute to the limit.

For the even indices, the Riemann-sum argument gives
\[
\begin{aligned}
N^{2q+1}
\sum_{\substack{|n|>N\\ n\ {\rm even}}}
|\widetilde F_n|^2
\longrightarrow
K_0^2(g_q^+-g_q^-)^2
\frac{(2/\pi)^{2q+2}}{2}
\int_{|u|>1}
u^{-2q-2}
\prod_{s\in E}
\left(1-\frac{u}{c_s}\right)^2du.
\end{aligned}
\]
The factor $1/2$ accounts for the density of the even integers.
With the change of variables $u=1/x$,
\[
\int_{|u|>1}
u^{-2q-2}
\prod_{s\in E}
\left(1-\frac{u}{c_s}\right)^2du
=
\int_{-1}^{1}
\prod_{s\in E}
\left(x-\frac{1}{c_s}\right)^2dx
=
J_E.
\]
Similarly,
\[
N^{2q+1}
\sum_{\substack{|n|>N\\ n\ {\rm odd}}}
|\widetilde F_n|^2
\longrightarrow
K_0^2(g_q^++g_q^-)^2
\frac{(2/\pi)^{2q+2}}{2}
J_O.
\]

Using $K_0^2=1/2$, we conclude that
\[
N^{2q+1}\|\widetilde R_{q,N}(f)\|^2
\longrightarrow
\frac{2^{2q}}{\pi^{2q+2}}
\left[
(g_q^+-g_q^-)^2J_E
+
(g_q^++g_q^-)^2J_O
\right].
\]
Finally, by \eqref{eq:AB},
\[
g_q^+-g_q^-=A+B,
\qquad
g_q^++g_q^-=A-B,
\]
so this is precisely \eqref{eq:L2}.
\end{proof}

\begin{remark}\label{rem:reduction}
For comparison, applying the same Riemann-sum argument to the
exact-jumps coefficients $F_n$ gives \eqref{eq:exact-L2}:
\[
\lim_{N\to\infty}N^{2q+1}\|R_{q,N}(f)\|^2
=
\frac{2^{2q+2}}{\pi^{2q+2}(2q+1)}(A^2+B^2).
\]
Thus, so long as the parity balance holds, computing the boundary
values from the generalized Fourier coefficients does not degrade the
rate of convergence, but changes the asymptotic constant.
\end{remark}

\section*{Acknowledgements}
R.~Barkhudaryan and G.~Gevorkyan were supported by the Higher Education and Science Committee of MESCS RA (Research project No.~25RG-1A195).

The authors used generative AI tools solely as an assistive aid during the preparation of this manuscript to generate some ideas and to improve the clarity, readability, language, and presentation of the text. The authors take full responsibility for the accuracy, originality, integrity, and final content of the manuscript, and all AI-assisted output was carefully reviewed and validated by the authors before inclusion.

\bibliographystyle{abbrv}
\bibliography{main}
\end{document}